\documentclass[12pt,a4paper,twoside]{article}
\usepackage{amssymb,amsthm}
\usepackage[namelimits,sumlimits,fleqn]{amsmath}
\usepackage{setspace}
\usepackage[a4paper,top=35mm, bottom=30mm, left=20mm, right=20mm]{geometry} 
\usepackage[small, margin=20pt]{caption}
\usepackage{float}\restylefloat{figure}
\usepackage{url,hyperref}

\usepackage{fancyhdr}
\allowdisplaybreaks[4]

\title{Pl\"unnecke's Inequality}
\author{Giorgis Petridis}
\date{}

\theoremstyle{plain}
\newtheorem{theorem}{Theorem}[section]
\newtheorem{lemma}[theorem]{Lemma}
\newtheorem{proposition}[theorem]{Proposition}
\newtheorem{corollary}[theorem]{Corollary}

\theoremstyle{definition}

\newtheorem{question}[theorem]{Question}

\newtheorem*{definition}{Definition}
\newtheorem*{acknowledgement}{Acknowledgement}

\theoremstyle{definition}
\newtheorem*{remark}{Remark}

\newcommand{\rst}[1]{\ensuremath{{\mathbin\upharpoonright}%
\raise-.5ex\hbox{$#1$}}} 

\DeclareMathOperator{\im}{Im} 

\begin{document}

\pagenumbering{arabic}

\setcounter{section}{0}

\bibliographystyle{plain}

\maketitle

\thispagestyle{plain}

\begin{abstract}
Pl\"unnecke's inequality is the standard tool to obtain estimates on the cardinality of sumsets and has many applications in additive combinatorics. We present a new proof. The main novelty is that the proof is completed with no reference to Menger's theorem or Cartesian products of graphs. We also investigate the sharpness of the inequality and show that it can be sharp for arbitrarily long, but not for infinite commutative graphs. A key step in our investigation is the construction of arbitrarily long regular commutative graphs. Lastly we prove a necessary condition for the inequality to be attained.
\end{abstract}

\section[Introduction]{Introduction}
\label{Introduction}

\vspace{1ex}

Pl\"{u}nnecke's inequality is among the most commonly used tools
in additive combinatorics. It was discovered by Helmut
Pl\"{u}nnecke in the late 1960s. The inequality puts bounds on the
magnification ratios of a directed, layered graph $G$, which are
defined as:
$$D_i(G) = \min_{\emptyset\neq Z\subseteq V_0} \frac{|\im^{(i)}(Z)|}{|Z|}.$$
Pl\"unnecke discovered that under some commutativity conditions on
graphs, which have since been known as Pl\"unnecke conditions and
will be defined later, the sequence $D_i^{1/i}(G)$ is decreasing. The directed layered graphs that obey these conditions are called commutative (or Pl\"unnecke) graphs. In particular Pl\"unnecke proved \cite{Plunnecke1970} the following.
\begin{theorem}[Pl\"{u}nnecke]\label{Plunnecke}
Let $G$ be a commutative graph with $D_h(G) = \Delta^h$. Then
$D_i(G)\geq \Delta^i$ for all $1\leq i \leq h$.
\end{theorem}
The main objective of the paper is to present a new proof of Theorem \ref{Plunnecke}.

Imre Ruzsa has simplified Pl\"unnecke's proof in \cite{Ruzsa1989,Ruzsa1991}. Pl\"unnecke's and Ruzsa's arguments have more similarities than
differences as their backbone is the same. Ruzsa's simplified approach has become the standard way to prove the inequality and we will thus use it as the point of comparison with the present argument. The reader should bear in mind that the same could have been achieved for Pl\"unnecke's proof. 

Ruzsa's argument relies on two key ingredients: Menger's theorem \cite{Menger} and Cartesian
products of graphs. While there are several variations in the literature
\cite{Malouf1995,Nathanson1996,Ruzsa2009,TaoNotes,Tao-Vu2006} they
all follow the original approach closely in first proving  the
special case when $D_h(G)=1$ by applying Menger's theorem and then
deducing the inequality by using Cartesian product of graphs. Here
we present an elementary and more direct proof, which stays close
to Pl\"unneke's and Ruzsa's argument for the special case, but uses neither of the
two ingredients.

Completing the proof with no reference to Menger's or any other
equivalent theorem is noteworthy for two reasons. It shows that
Pl\"{u}nnecke's inequality is a direct consequence of
Pl\"{u}nnecke's conditions and little else. Therefore the bounds
on the cardinality of sumsets that follow from it are also a
direct consequence of commutativity of addition and little else.
The second reason is that by avoiding Menger's theorem we are able
to complete the proof without using Cartesian products of graphs. It was so far not clear whether this very helpful tool was a necessary ingredient and removing it makes the proof more transparent.

Despite its widespread use there has so far been no attempt to
investigate whether Pl\"unnecke's inequality is sharp. We answer
this question for both finite and infinite commutative graphs.
\begin{theorem}\label{Sharp Commutative for Finite}
For all $C\in \mathbb{Q}$ and $h\in \mathbb{Z}^+$ there exists a
commutative graph with $$ D_i(G) = C^i$$ for all $1\leq i \leq
h$.
\end{theorem}

\begin{theorem}\label{Sharp Commutative for Infinite}
Let $G$ be an infinite commutative graph. Then
$$ D_i(G) = C^i$$ can hold if and only if $C=1$.
\end{theorem}
The extremal graphs for Pl\"unnecke's inequality we present are all regular. It is natural to ask whether this condition is necessary. The final result of this paper is to show that in a way it is: every commutative graph where Pl\"unnecke's inequality is attained must contain a regular commutative subgraph - the exact meaning of this assertion is explained in
Section \ref{Inverse}. 

The remaining sections of the paper as organised as follows. In
Section \ref{Commutative Graphs} we introduce commutative graphs and
the notation used at the remainder of the paper. Section
\ref{Proof} is devoted to the proof of Pl\"unnecke's inequality;
an entirely self-contained argument is found in Sections
\ref{Weighted Commutative Graphs} and \ref{Separating Sets}. In Section \ref{Regular} we
prove Theorems \ref{Sharp Commutative for Finite} and \ref{Sharp
Commutative for Infinite}. Finally in Section \ref{Inverse} we deduce the existence of
the regular subgraph in the case when all the quantities $D_i^{1/i}(G)$ are equal. 
\begin{acknowledgement}
The author would like to thank Tim Gowers for suggesting looking
for a proof of Pl\"unnecke's inequality that does not use
Cartesian products of graphs and for sharing his insight. In
particular, the ideas of working with weighted commutative graphs
was his. He would also like to thank Ben Green, Peter Keevash and Imre Ruzsa for many helpful suggestions. 
\end{acknowledgement}

\section[Commutative Graphs]{Commutative Graphs}
\label{Commutative Graphs}

\vspace{1ex}

The material in this section can be found in any of the standard
references \cite{Nathanson1996,Ruzsa2009,Tao-Vu2006}. The notation
used is however slightly different.

\subsection{Commutative graphs: definition and notation}

\label{Commutative graphs: definition and notation}

$G$ will always be a directed layered graph with edge set $E(G)$
and vertex set $V(G)=V_0\cup\dots\cup V_h$, where the $V_i$ are
the \textit{layers} and $h$ the \textit{level} of the graph. For
any $S\subseteq V_i$ we will write $S^c= V_i \backslash S$ for the
complement of $S$ in $V_i$ and not in $V(G)$. We will furthermore
assume that directed edges exist only between $V_i$ and $V_{i+1}$
and denote this set of edges by $E(V_i,V_{i+1})$.

In order to introduce Pl\"unnecke's conditions we briefly recall
that given an integer $k$ and a bipartite undirected graph
$G(X,Y)$ we say that a \textit{one-to-$k$ matching} exits from $X$
to $Y$ if we can find distinct elements $\{ y^i_x : x\in
X\,\mbox{and}\,1\leq i\leq k\}$ in $Y$ such that $xy^i_x\in E(G)$
for all $x\in X$ and $1\leq i \leq k$. A one-to-one matching is
referred to as a \textit{matching}.

Pl\"{u}nnecke's \textit{upward} condition states that if $uv\in
E(G)$, then there exists a matching from $\im(v)$ to $\im(u)$ (in
the bipartite graph $G(\im(u),\im(v))$ where $xy$ is an undirected
edge if and only if it is a directed edge in $G$). Pl\"{u}nnecke's
\textit{downward} condition states that if $vw\in E(G)$, then then
there exists a matching from $\im^{-1}(v)$ to $\im^{-1}(w)$ (in
the bipartite graph $G(\im^{-1}(v),\im^{-1}(w))$ where $xy$ is an
undirected edge if and only if it is a directed edge in $G$). A \textit{commutative graph} is a directed layered graph that
satisfies both properties. In the literature such graphs are sometimes referred to as \emph{Pl\"unnecke graphs}.

The most typical example is $G_+(A,B)$, the \textit{addition
graph} of two sets $A$ and $B$ in a commutative group.
This is defined as the directed graph whose $i$th layer $V_i$ is
$A+iB$ and a directed edge exists between $x\in V_{i-1}$ and $y\in
V_i$ if and only if $y-x\in B$. When we take $A=\{0\}$ and
$B=\{\gamma_1,\ldots, \gamma_n\}$ where 0 is the identity and
$\gamma_i$ the generators of a free commutative group we call
$G_+(\{0\},\{\gamma_1,\ldots,\gamma_n\})$ the \textit{independent
addition} graph on $n$ generators.

$d^+_H(v) =|\{w : vw \in E(H) \}|$ is the \textit{outgoing degree}
of a vertex $v$ in a subgraph $H$ and $d^+(v)=d^+_G(v)$ is the
outgoing degree of $v$ in $G$. In particular $d^+_H(v)\leq
d^+_G(v)$. Similarly, $d^-_H(v) =|\{u : uv \in E(H) \}|$ is the
\textit{incoming degree} of a vertex $v$ in a subgraph $H$ and
$d^-(v)=d^-_G(v)$ is the incoming degree of $v$ in $G$. In
particular $d^-_H(v)\leq d^-_G(v)$.

A \textit{path} of length $l$ in $G$ is a sequence of vertices
$v_1,\ldots, v_l$ so that $v_{i}v_{i+1}\in E(G)$ for all $1\leq i
\leq l-1$. For a subgraph $H$ of $G$ and $Z\subseteq V(H)$ we thus define
\begin{eqnarray*}
\im^{(i)}_{H}(Z) & = & \{v\in V(H) : \exists \hspace{4 pt}
\mbox{path of length $i$ in $H$ that starts in $Z$ and ends in
$v$\}}
\end{eqnarray*}
and
\begin{eqnarray*} \im^{(-i)}_{H}(Z)& = &\{v\in V(H) : \exists
\hspace{4 pt} \mbox{path of length $i$ in $H$ that starts in $v$
and ends in $Z$} \}.
\end{eqnarray*}
When the subscript is omitted we are taking $H$ to be $G$. When
$i=1$, and consequently $\im^{(1)}(Z)$ is the neighbourhood of $Z$
in $H$, the superscript will be omitted. We can now formally
define magnification ratios. As we have seen the $i$th
magnification ratio of $G$ is defined as
$$ D_i(G) = \min_{\emptyset\neq Z\subseteq V_0} \frac{|\im^{(i)}(Z)|}{|Z|}. $$
We will also write $$\Delta= D_h^{1/h}(G).$$ For $X,Y\subseteq V(G)$ the \emph{channel between $X$ and $Y$} is the subgraph that consists of directed paths starting at $X$ and finishing in $Y$. For $Z\subseteq V_0$ the \textit{channel of $Z$} is the channel between $Z$ and $V_h$. 

A \textit{separating set} in any subgraph
$H$ is a set $S\subseteq V(H)$ that intersects all directed paths of maximum length in $H$.

\subsection{Properties of commutative graphs}

\label{Properties of Commutative Graphs}

The following properties of commutative graphs are standard and
will be used repeatedly.

(1) For $i>j$ and $X\subseteq V_i$, $Y\subseteq V_j$ the channel between $X$ and $Y$ is a commutative graph in its own right. An important special case is the channel of $Z\subseteq V_0$.

(2) For $uv\in E(G)$ Pl\"{u}nnecke's conditions imply $d^+(u)\geq
d^+(v)$ and $d^-(u)\leq d^-(v)$.

(3) For commutative graphs $G$ and $H$ we define their Cartesian
product $G\times H$ as follows.  The $i$th layer of $G \times H$ is the Cartesian
product of the $i$th layer of $G$ with the $i$th layer of $H$. As
for the edges, $(u,x)(v,y)\in E(G\times H)$ if and only if $uv\in
E(G)$ and $xy\in E(H)$. $G\times H$ is a commutative graph with 
$D_i(G\times H) = D_i(G) \, D_i(H)$. Vertex degrees are also
multiplicative as $d^\pm_{G \times H} ((u,x)) = d^\pm_G(u) \,
d^\pm_H(x)$.

(4) We define the inverse $I$ of a commutative graph $G$ as
follows: the $i$th layer of $I$ is the $(h-i+1)$th layer of $G$
and $uv\in E(I)$ if and only if $vu\in E(G)$. One can informally
think of $I$ as the graph consisting of all paths from $V_h$ to
$V_0$. $I$ is always a commutative graph due to the symmetry of
Pl\"{u}nnecke's conditions.

\subsection{Hall's marriage theorem}

\label{Hall's marriage theorem}

We finish this introductory section by stating Hall's marriage
theorem for bipartite graphs $G=G(X,Y)$. For any $x\in X$ we
define its neighborhood by
$$ \Gamma(x) = \{y \in Y : xy\in E(G) \} $$
and the neighborhood of $S\subseteq X$ by
$$ \Gamma(S) = \bigcup_{x\in S} \Gamma(x) .$$

It is clear that in order to have a one-to-$k$ matching from $X$
to $Y$ we need $|\Gamma(S)|\geq k \,|S|$ for all $S\subseteq X$.
Philip Hall proved in 1935 that the converse is also true
\cite{Hall}.
\begin{lemma}[Hall]\label{Hall}
Let $G(X,Y)$ be a bipartite graph. Then a one-to-$k$ matching
exists from $X$ to $Y$ if and only if $$|\Gamma(S)|\geq k \,
|S|~\mbox{for all}~ S \subseteq X .$$
\end{lemma}

\section[Proof of Pl\"{u}nnecke's Inequality]{Proof of Pl\"{u}nnecke's Inequality}
\label{Proof}

\vspace{1ex}

We begin our examination of Pl\"{u}nnecke's inequality with a new
proof of Theorem \ref{Plunnecke}. The proof is inspired by the
work of Ruzsa that appeared in \cite{Ruzsa1989, Ruzsa1991} and in
particular by an exposition of Ruzsa's argument due to Terence Tao
\cite{TaoNotes}. However, there are crucial differences, as
Menger's theorem and Cartesian products of graphs are not needed.

\subsection{Outline of the Pl\"unnecke-Ruzsa proof}

\label{Outline of Ruzsa's Proof}

The traditional proof of Theorem \ref{Plunnecke} can be split in two distinct parts. The first is to
establish the special case when $\Delta= 1$. The key is the relation between magnification ratios and separating sets in the graph. By applying Menger's theorem Pl\"unnecke proved
the following powerful result:
\begin{proposition}[Plunnecke]\label{Ruzsa}
Let $G$ be a commutative graph with $D_h(G)\geq 1$. Then there
are $|V_0|$ vertex disjoint paths from $V_0$ to $V_h$ in $G$ and
therefore $D_i(G)\geq 1$ for all $i$.
\end{proposition}
The duality between separating sets and vertex disjoint paths is
exploited fully. This poses a serious obstacle when trying to
extend this idea for general values of $\Delta$ as Menger's
theorem is no longer useful. Even for integer $\Delta \neq 1$
there is an example which shows that proving the following natural
and plausible generalisation
\begin{question}\label{Ruzsa generalisation conjecture}
Suppose that $D_h(G) \geq k^h$ for some integer $k$. Then there
are $|V_0|$ vertex disjoint trees each having at least $k^i$
vertices in $V_i$.
\end{question}
would require ideas beyond those found in this paper.

The second part of the proof is to overcome this obstacle by
deducing the general case from Proposition \ref{Ruzsa}. Ruzsa achieved this using the
multiplicativity of magnification ratios. The quickest way to do
this is by using some graphs we will introduce in Section
\ref{Regular} (c.f. Section \ref{Regular}).
For any rational $q \leq \Delta$ there is a commutative graph
$R_q$ with $D_i(R_q)=q^{-i}$ for all $i=1,\dots,h$. We know that
$D_h(G\times R_q) = D_h(G)\,D_h(R_q) = (\Delta q^{-1})^h\geq 1$
and so $$ 1 \leq D_i(G\times R_q) = D_i(G)\,D_i(R_q) = D_i(G)
q^{-i} $$ This implies that $D_i(G)\geq q^i$ for all rationals
$q\leq \Delta$ and hence that $D_i(G)\geq \Delta^i$. For the
reader's benefit we will note that the standard deduction uses
independent addition graphs instead. In this context it is
mandatory to take the product of $r$ copies of $G$ with suitably
chosen independent addition graphs and then let $r\rightarrow
\infty$.

Ruzsa's approach is elegant, but leaves one question unanswered:
what is the precise role of Cartesian products in the proof and
how does it allow us to use Proposition \ref{Ruzsa} in such a
simple way when proving a generalisation is tricky? A simpleminded
approach is to see what the existence of the paths in $G\times
R_q$ implies about $G$, but this yields a mere reformulation of
Pl\"{u}nnecke's inequality. A more refined approach suggested by
Tim Gowers is to work in a weighted version of $G$. In this setting Menger's theorem could be replaced by the Max Flow - Min Cut theorem. 

In fact Theorem \ref{Plunnecke} will be
proved by focusing on the minimum cut in (the network generated
by) $G$ without using any properties of a maximal flow. In doing
so we will mirror Pl\"unnecke's proof of Proposition \ref{Ruzsa}
closely, but will introduce a further ingredient in Section
\ref{Separating Sets} that allows us to apply his argument to all
$\Delta$.

\subsection{Weighted Commutative Graphs}

\label{Weighted Commutative Graphs}

The proof of Proposition \ref{Ruzsa} is built around the fact that
when $\Delta =1$ there is a very natural relation between
separating sets in $G$ and magnification ratios. In order to make
use of this observation for general $\Delta$ we need to work with
weighted commutative graphs; i.e. a commutative graph with a
weight function
$$ w: V(G) \mapsto \mathbb{R}^+ .$$
We will eventually give every vertex in $V_i$ weight $\Delta^{-i}$.
The reasons behind this choice will become apparent shortly, but
different weights may be more suitable in other applications. It
should be noted that this can be thought of as an alternative to
taking a Cartesian product of $G$ with the $R_q$. We also need a
notion of the weight of a set of vertices in $G$ and so we define
the weight of any set $S\subseteq V(G)$ as
$$ w(S) = \sum_{v\in S} w(v). $$
In what follows this will equal
$$ \sum_{i=0}^h |S\cap V_i| \,C^{-i}$$
for a positive constant $C$. The heart of the proof of Proposition
\ref{Ruzsa} is to ``pull down'' any minimum separating set to one
that lies entirely in $V_0\cup V_h$. Pl\"unnecke achieved this by applying Pl\"{u}nnecke's
conditions to the paths given by Menger's theorem. The same can be done for weighted
commutative graphs and in fact without any reference to Menger's
or some other equivalent theorem. The following result
demonstrates how powerful Pl\"{u}nnecke's conditions are.
\begin{lemma}\label{Pull Down}
Let $C$ be a positive real and $G$ a weighted commutative graph with
vertex set $V_0\cup V_1\cup\dots\cup V_h$ and $w(v)=C^{-i}$ for
all $v\in V_i$. A separating set of minimum weight that lies entirely in $V_0\cup V_h$ exists.
\end{lemma}
Proving this lemma will be the main objective of the next
subsection. For the time being let us quickly see how to deduce
Theorem \ref{Plunnecke} from it.
\begin{corollary}\label{Min Weight}
Let $G$ a weighted commutative graph with vertex set $V_0\cup
V_1\cup\dots\cup V_h$ and $w(v)=\Delta ^{-i} = D_h(G)^{-i/h}$ for
all $v\in V_i$. The weight of any minimal separating set is
$|V_0|$.
\end{corollary}
\begin{proof}
By applying Lemma \ref{Pull Down} we can assume that $S_0\cup S_h$ is a
separating set of minimum weight with $S_i\subseteq V_i$. We know that
$\im^{(h)}(S^c_0)\subseteq S_h$ and so $|S_h| \geq
|\im^{(h)}(S^c_0)|\geq D_h(G) |S_0^c|$. This in turn implies
$w(S)= w(S_0) + w(S_h) = |S_0| + |S_h| \,D_h^{-1}(G) \geq |S_0|+
|S_0^c| = |V_0|$. On the other hand $V_0$ is a separating set and
hence $w(S)=|V_0|$ for any separating set of minimum weight.
\end{proof}

Pl\"{u}nnecke's inequality follows in a straightforward manner:

\begin{proof}[Proof of Theorem \ref{Plunnecke}]
We consider any $Z\subseteq V_0$ in the weighted version of $G$,
where each $v\in V_i$ has weight $\Delta^{-i}$. $Z^c \cup
\im^{(i)}(Z)$ is a separating set and thus
$$ |V_0| \leq w(Z^c \cup \im^{(i)}(Z)) = w(Z^c) + w(\im^{(i)}(Z)) =
|V_0| - |Z| + |\im^{(i)}(Z)| \,\Delta^{-i}$$ I.e.
$|\im^{(i)}(Z)|\geq\Delta^i |Z|$. Taking the minimum over all
$Z\subseteq V_0$ gives the lower bound on $D_i(G)$.
\end{proof}

\subsection{Separating sets on weighted commutative graphs}

\label{Separating Sets}

We now turn to the proof of Lemma \ref{Pull Down}.
The key is to make optimal use of separating sets of minimal
weight. Instead of using vertex disjoint paths we rely on
the following elementary observation. Suppose that $S$ is a separating
set of minimum weight. Then for any $Z\subseteq S$
$$w(\im(Z))\geq w(Z)~~\mathrm{and}~~w(\im^{-1}(Z))\geq w(Z).$$
We begin with establishing the simplest case of Lemma \ref{Pull Down} when
$h=2$ and the middle layer is the separating set. We will need to
apply the following in the coming section and therefore state it
in slightly more general terms.
\begin{lemma}\label{Technical Three Layer Pull Down}
Let $C$ be a positive real and $H$ be a commutative graph of level two
with vertex set $U_0\cup U_1\cup U_2$. Suppose that for all
$S\subseteq U_1$
$$ |\im(S)|\geq C |S| ~~\mathrm{and}~~ |\im^{-1}(S)|\geq C^{-1} |S|. $$
If $X_i$ is the set of vertices in $U_1$ that have incoming degree
equal to $i$ and $Y_i$ is set of vertices in $U_2$ that have
incoming degree equal to $i$, then
$$ C |X_i| = |Y_i|.$$
Similarly if $X^\prime_i$ is the set of vertices in $U_1$ that
have outgoing degree equal to $i$ and $Y^\prime_i$ is the set of
vertices in $U_0$ that have outgoing degree equal to $i$, then
$$ C^{-1} |X^\prime_i| = |Y^\prime_i|.$$
\end{lemma}
\begin{proof}
The sets $X_i$ form a partition of $U_1$. We partition $U_2$ into:

$T_k =  \im (X_k)$ \\
$T_{k-1} = \im(X_{k-1}) \backslash T_k$ \\
$\vdots$ \\
$T_1 = \im(X_1) \backslash (T_2\cup \dots \cup T_{k})$

Similarly we have a partition of $U_1$ into $X^\prime_1,\dots,
X^\prime_{k^\prime}$ and a partition of $U_0$ into:

$T^\prime_{k^\prime} =  \im^{-1}(X^\prime_{k^\prime})$ \\
$T^\prime_{k^\prime-1} = \im^{-1}(X^\prime_{k^\prime-1}) \backslash T_{k^\prime}$ \\
$\vdots$ \\
$T^\prime_1 = \im^{-1}(X^\prime_1) \backslash (T^\prime_2\cup
\dots \cup T^\prime_{k^\prime})$

This is probably a good moment for the reader to look at Figure
\ref{Figure of Technical Lemma}.

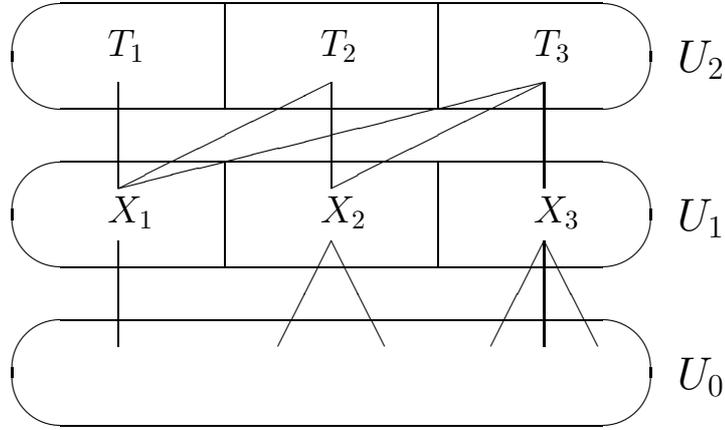
\begin{figure}[ht]
\begin{center}
\setlength{\unitlength}{0.7 cm}
 \begin{picture}(16,8)
  \put(14.5,6.7){\Large{$U_2$}}
  \put(8,7){\oval(12,2)}
  \put(6,6){\line(0,1){2}}
  \put(10,6){\line(0,1){2}}
  \put(3.8,7){\large{$T_1$}}
  \put(7.8,7){\large{$T_2$}}
  \put(11.8,7){\large{$T_3$}}
  \multiput(4,4.5)(4,0){3}{\line(0,1){2}}
  \put(4,4.5){\line(2,1){4}}
  \put(4,4.5){\line(4,1){8}}
  \put(8,4.5){\line(2,1){4}}
  \put(14.5,3.7){\Large{$U_1$}}
  \put(8,4){\oval(12,2)}
  \put(6,3){\line(0,1){2}}
  \put(10,3){\line(0,1){2}}
  \put(3.8,3.85){\large{$X_1$}}
  \put(7.8,3.85){\large{$X_2$}}
  \put(11.8,3.85){\large{$X_3$}}
  \put(4,3.5){\line(0,-1){2}}
  \put(8,3.5){\line(1,-2){1}}
  \put(8,3.5){\line(-1,-2){1}}
  \put(12,3.5){\line(0,-1){2}}
  \put(12,3.5){\line(1,-2){1}}
  \put(12,3.5){\line(-1,-2){1}}
  \put(14.5,0.7){\Large{$U_0$}}
  \put(8,1){\oval(12,2)}
 \end{picture}
\end{center}
\caption{\label{Figure of Technical Lemma}\small{An illustration
for the $k=3$ case.}}
\end{figure}
By the definition of the $T_i$ we have that $$\im (X_j\cup \dots
\cup X_k ) = T_j \cup\dots\cup T_k.$$ If we let $x_i = |X_i|$
and $t_i = |T_i|$, then the hypothesis on $H$ implies that
$$ \sum_{i=j}^k t_i \geq C \sum_{i=j}^k x_i \qquad\mbox{for all} \quad 1\leq j \leq k .$$
Adding these inequalities for $j=1, \dots , k$ gives
$$ \sum_{i=1}^k i t_i \geq C \sum_{i=1}^k i x_i .$$
It follows from Pl\"{u}nnecke's downward condition and the definition of $T_i$ and $X_i$ that
$d^-(v)\geq i$ for all $v\in T_i$. Hence
\begin{eqnarray*}
|E(U_0,U_1)|& = & \sum_{i=1}^{k} |E(U_0,X_i)| \\
            & = & \sum_{i=1}^{k} i x_i \\
            & \leq & C^{-1} \sum_{i=1}^{k} i t_i \\
            & \leq & C^{-1} \sum_{i=1}^{k} |E(U_1,T_i)| \\
            & = & C^{-1} |E(U_1,U_2)|
\end{eqnarray*}
We repeat the above calculation this time using the second partition of $U_1$
and get
\begin{eqnarray*}
|E(U_1,U_2)|& \leq C |E(U_0,U_1)|.
\end{eqnarray*}
Putting everything together yields:
$$ |E(U_0,U_1)| \leq C^{-1} |E(U_1,U_2)| \leq |E(U_0,U_1)| .$$
We must therefore have equality in every step, which implies that
$Y_i=T_i$ and $Y^\prime_i=T^\prime_i$, as well as $C |X_i|=|Y_i|$
and $C^{-1} |X^\prime_i|=|Y^\prime_i|$.
\end{proof}
We now apply the lemma to ``pull down'' minimal separating sets in
the special, yet important, class of graphs of level two discussed
in the beginning of the subsection.
\begin{lemma}\label{Three Layer Pull Down}
Let $C$ be a positive real and $H$ be a weighted commutative graph of
level two with vertex set $U_0\cup U_1\cup U_2$ and $w(v)=C^{-i}$
for all $v\in V_i$. Suppose that $U_1$ is a separating set of
minimum weight. Then so is $U_0$.
\end{lemma}
\begin{proof}
For every $S\subseteq U_1$ both $S^c\cup \im(S)$ and $S^c \cup
\im^{-1}(S)$ are separating sets. The minimality of $w(U_1)$
implies that
$$ |\im(S)|\geq C|S|~~~\mathrm{and}~~~ |\im^{-1}(S)|\geq C^{-1}|S|.$$
We can therefore apply Lemma \ref{Technical Three Layer Pull Down}
to get
$$w(U_1) = C^{-1} |U_1| = C^{-1} \left|\bigcup_{i=1}^{k^\prime} X^\prime_i \right| =
C^{-1} \sum_{i=1}^{k^\prime} |X^\prime_i| = \sum_{i=1}^{k^\prime}
|Y^\prime_i| = \left| \bigcup_{i=1}^{k^\prime} Y^\prime_i \right|
= |U_0| = w(U_0). \qedhere$$
\end{proof}
We are finally able to prove Lemma \ref{Pull Down}, which will
finish the proof of Theorem \ref{Plunnecke}.
\begin{proof}[Proof of Lemma \ref{Pull Down}]
Let $S$ be any separating set of minimum weight and $S_i = S\cap
V_i$. Let $j\in \{0,1,\dots,h-1\}$ be maximal subject to $S_j\neq
\varnothing$. We will show that when $j>0$ we can find another
separating set of minimum weight that lies in $V_0\cup\dots\cup
V_{j-1}\cup V_h$.

We work in a subgraph $H$ of level two consisting of all paths in
$G$ that start in a suitably chosen $U_0\subseteq V_{j-1}$ and end
in a suitably chosen $U_2\subseteq V_{j+1}$. $U_0$ consists of all
vertices in $V_{j-1}$ that can be reached via paths in $G$ that
successively pass from $S_0^c,\dots, S_{j-1}^c$ and $U_2$ consists
of all vertices in $V_{j+1}$ that lead to $S_h^c$. $S$ is a
separating set of minimal weight and thus the middle layer $U_1$
equals $S_j$. In the weighted version of $H$, where vertices in
$U_i$ have weight $C^{-i}$, $U_1$ is a separating set of minimum
weight (if not let $S_j^\prime$ be a separating set of smaller
weight and observe that $S_0\cup\dots\cup S_{j-1}\cup
S_j^\prime\cup S_h$ is then a separating set in $G$ of smaller
weight than $S$). By Lemma \ref{Three Layer Pull Down} $U_0$ is
also a separating set of minimum weight in $H$ and thus
$S_0\cup\dots\cup S_{j-1}\cup U_0\cup S_h$ is a separating set of
minimum weight in $G$.
\end{proof}
Looking back at the proof of Pl\"unnecke's inequality we realise
that Pl\"unnecke's conditions were not used directly. Instead we
relied on two properties that follow from them: properties (1) and
(2) in Section \ref{Commutative Graphs}. It is clear
that both are necessary in the proof. It is therefore natural to
ask how different this pair of conditions is compared to
Pl\"unnecke's.

Ruzsa has already noted in \cite{Ruzsa2009} that the two sets of conditions are equivalent and as a consequence the proof of Pl\"unnecke's inequality requires the full strength of Pl\"unnecke's conditions. This observation was left as an exercise and so we offer a quick explanation. Suppose that Pl\"unnecke's, say upward, condition fails for an edge $uv$. It
follows that there is no matching from $\im(v)$ to $\im(u)$ in the
bipartite graph $G(\im(v),\im(u))$ where $xy$ is an edge if and
only if $yx$ is an edge in $G$. By Lemma \ref{Hall} we know there
exists $S\subseteq\im(v)$ such that
$|\im^{(-1)}(S)|<|S|$. Now consider the channel $H$
between $u$ and $S$. This is a commutative graph and $uv\in E(H)$,
yet $d^+_H(u)=|\im^{(-1)}(S)|<|S|=d^+_H(v)$.

Before moving on we prove a slight variation of Lemma
\ref{Technical Three Layer Pull Down}, which will be useful in Section \ref{Inverse}.
\begin{lemma}\label{Technical Three Layer Pull Down Prime}
Let $C$ be a positive real and $H$ be a commutative graph of level two
with vertex set $U_0\cup U_1\cup U_2$. Suppose that for all
$S\subseteq U_1$ we have
$$|\im^{-1}(S)|\geq C^{-1} |S| ~~\mbox{and}~~C |E(U_0,U_1)| = |E(U_1,U_2)|.$$ Then $|U_1|= C |U_0|$.
\end{lemma}
\begin{proof}
This is almost identical to what we have already seen. We
partition $U_1$  and $U_0$ into respectively $X^\prime_1,\dots,
X^\prime_{k^\prime}$ and $T^\prime_1,\dots, T^\prime_{k^\prime}$
like in the proof of Lemma \ref{Technical Three Layer Pull Down}.
We have
$$\im^{-1}(X^\prime_j\cup \dots \cup X^\prime_{k^\prime} ) = T^\prime_j \cup\dots\cup T^\prime_{k^\prime}.$$
If we once again let $x^\prime_i = |X^\prime_i|$ and $t^\prime_i =
|T^\prime_i|$, then the first hypothesis on $H$ implies that
\begin{eqnarray}\label{Partial Sum Prime}
C \sum_{i=j}^{k^\prime} t^\prime_i \geq \sum_{i=j}^{k^\prime}
x^\prime_i \qquad\mbox{for all} \quad 1\leq j \leq k^\prime.
\end{eqnarray}
Adding the $k^\prime$ inequalities gives
\begin{eqnarray*}
C \sum_{i=1}^{k^\prime} i t^\prime_i \geq \sum_{i=1}^{k^\prime} i
x^\prime_i.
\end{eqnarray*}
From Pl\"{u}nnecke's upward condition we know that $d^+(v)\geq i$
for all $v\in T^\prime_i$ and in a similar fashion to the proof of
Lemma \ref{Technical Three Layer Pull Down} we get
\begin{eqnarray*}
|E(U_1,U_2)| & \leq & C |E(U_0,U_1)|.
\end{eqnarray*}
The second condition on $H$ implies that equality must hold in
every step. In particular setting $j=1$ on \eqref{Partial Sum
Prime} gives
$$ C|U_0| = C \sum_{i=1}^{k^\prime} t^\prime_i = \sum_{i=1}^{k^\prime} x^\prime_i = |U_1|. \qedhere$$
\end{proof}

\section[Regular Commutative Graphs]{Regular Commutative Graphs}
\label{Regular}

\vspace{1ex}
We now turn to investigating the sharpness of Pl\"unnecke's inequality and prove Theorems \ref{Sharp Commutative for Finite} and \ref{Sharp Commutative for Infinite}. For the former we construct arbitrarily long commutative graphs where $D_i^{1/i}(G)$ is constant. The latter will be proved by examining the growth of commutative graphs that originate at a singleton. 

\subsection{Regular commutative graphs}

\label{Regular Commutative Graphs}

The two theorems are closely related with the existence of regular commutative graphs.
\begin{definition}
Let $C\in \mathbb{Q}^+$. $R_{C}$ is a \textit{regular commutative graph of ratio $C$} whenever $d^-(v)=d$ and $d^+(v)=C d$ for all $v\in
V(G)$ and some $d\in \mathbb{Z}^+$.
\end{definition}
It is easy to see why they are important in this context.
\begin{lemma}\label{regular-sharp}
Let $C\in \mathbb{Q}^+$ and $i\leq h$ be positive integers.
Suppose that $G$ is a regular commutative graph of ratio $C$ with vertex set $V_0\cup\dots\cup V_h$. Then
$$D_i(G) = C^i$$
and
$$ |V_i| = C^i |V_0|.$$
Furthermore the inverse of $G$ is an $R_{C^{-1}}$.
\end{lemma}
\begin{proof}
Suppose that $d^-=d$ and $d^+=Cd$ for all vertices of the graph. There are $C d |Z|$ edges coming out from every $Z
\subseteq V_0$. These edges land in at least $C |Z|$ vertices in
$V_1$ and hence we get that $|\im(Z)| \geq C|Z|$ -- and
consequently that $D_1(G)\geq C$. Looking at
$\im^{(i)}(Z)=\im(\im^{(i-1)}(Z))$ we see that $|\im^{(i)}(Z)|
\geq C^i|Z|$ -- and consequently that $D_i(G)\geq C^i$. Next we
count the edges between $V_{i-1}$ and $V_i$ in two different ways
to get
$$ C d\,|V_{i-1}| =|E(V_{i-1},V_i)| = d\, |V_i|.$$
Hence $|V_i| = C^i |V_0|$, which shows that $D_i(G)\leq C^i$.

We know that the inverse of $G$ is a commutative graph. It is
furthermore regular with ratio $C^{-1}$.
\end{proof}
To prove Theorem \ref{Sharp Commutative for Finite} it is therefore enough to construct arbitrarily long $R_C$ for all $C\in \mathbb{Q}^+$. Let us begin by two simple yet fundamental observations. It is enough to construct arbitrarily long $R_{k}$ for all positive integers $k$
because if we let $C=p/q$ be any rational, then the Cartesian product of an $R_{p}$ with the inverse of an $R_{q}$ is an $R_{C}$. A path is an infinite $R_1$ so from now on we will focus on $R_{k}$ for integer $k>1$.

\subsection{Arbitrarily long regular commutative graphs}

\label{Arbitrarily R_{k}}

We begin with the explicit construction of arbitrarily long $R_k$. Getting an $R_{k}$ of level two is not hard, but we will not
present the simplest example as it cannot be extended to an
$R_{k}$ of level three. We will instead inductively build
arbitrarily long $R_{k}$. Our aim is to take an $R_{k}$ of
level $h$ and add a layer from below in such as a way as to get an
$R_{k}$ of level $h+1$. To achieve this we have to tweak the
$R_{k}$ of level $h$ slightly by taking its Cartesian product
with a suitably chosen commutative graph. The following graph has
the desired properties.
\begin{lemma}\label{One to k map}
Let $k$ and $h$ be positive integers. There exists an $R_{1}$ of level
$h$ that gives rise to a one-to-$k$ matching from the image of any
$v\in U_0$ to $U_0$ itself. $U_0$ is the bottom layer
of the graph.
\end{lemma}
\begin{proof}
We work in $\mathbb{Z}_{2k^2}$ and consider the level $h$ addition
graph $G_+(A,B)$  for $A =\mathbb{Z}_{2k^2}$ and
$$B=\{0,1,2,\ldots, k-1,k,2k,3k,\ldots,k^2\}.$$
This is an $R_{1}$. We define a map $\theta$ from the image of
any $v\in U_0$ to $U_0^k$  by:

$\theta(v+0) = \{v,v-1,v-2,\ldots,v-(k-1)\} \\
\theta(v+1) = \{v+1,v+1-2k,v+1-3k,\ldots,v+1-k^2\} \\
\theta(v+2) = \{v+2,v+2-2k,v+2-3k,\ldots,v+2-k^2\} \\
\vdots \\
\theta(v+k-1) = \{v+k-1,v+(k-1)-2k,v+(k-1)-3k,\ldots,v+(k-1)-k^2\} \\
\theta(v+k) = \{v+k,v-k,v-2k,\ldots,v-(k-1)k\} \\
\theta(v+2k) = \{v+2k,v+2k-1,v+2k-2,\ldots,v+2k-(k-1)\} \\
\vdots \\
\theta(v+k^2) = \{v+k^2,v+k^2-1,v+k^2-2,\ldots,v+k^2-(k-1)\}.$

A routine check confirms that every element of $\theta(v+j)$ is
indeed joined to $v+j$ in the graph. For example, $v+1$ is joined
to $v+1$ as it equals $v+1-0$ and $v-k$ is joined to $v+k$ as it
equals $v+k-2k$. A second routine check confirms that
$\theta(v+b)\cap \theta(v+b^\prime) = \emptyset$ for all $v\in
U_0$ and distinct $b,b^\prime \in B$. In other words the graph
yields a one-to-$k$ matching between the image of any $v\in U_0$
and $U_0$ itself, as claimed.
\end{proof}
We can now complete the inductive step by combining the above with
Lemma \ref{Hall} and the multiplicativity of degrees.
\begin{proposition}\label{Multilayer Extension}
Suppose that an $R_k$ of level $h$ exists with the property that every
vertex in the first layer is joined to every vertex in the second
layer. Then an $R_{k}$ of level $h+1$ with the same property exists.
\end{proposition}
\begin{proof}
Suppose that $W_0,\dots,W_{h}$ are the layers of $R_k$ with
$|W_0|=d$. The defining properties of $R_k$ imply that $d^-=d$ and $d^+= d k$. Let $R_1$ be the graph described in Lemma \ref{One
to k map} with layers $U_0,\dots,U_{h}$.

We let $G_h = R_k \times R_1$. This graph is a regular commutative graph of ratio $k$ whose
bottom layer has size $|W_0\times U_0|=2 d k^2 $. Next we add a
layer of size $2 d k$ to the bottom and join every added vertex to
the whole of $W_0\times U_0$. Let $G_{h+1}$ be the resulting graph
of level $h+1$, which is regular with incoming degree $2 d k$ and
outgoing $2 d k^2$. To complete the proof we show that $G_{h+1}$
is a commutative graph.

We only need to check Pl\"{u}nnecke's conditions involving the
recently added bottom layer. The remaining layers pose no problem
as they belong to $G_h$, which is commutative. The downward is immediate as the size of
the bottom layer equals the incoming degree. To check the upward
we consider an edge $u(w,v)$, where $u$ lies in the bottom layer
of $G_{h+1}$ and $(w,v)$ lies in the second layer; i.e. $w\in W_0$
and $v\in U_0$. Pl\"{u}nnecke's condition requires finding a
matching from $\im_{G_{h+1}}((w,v))=W_1\times \im_{R_1}(v)$ to
$\im_{G_{h+1}}(u) = W_0\times U_0$.

With this in mind we turn our attention to the bipartite graph
$(W_1\times \im_{R_1}(v) , W_0\times U_0)$ and aim to apply Lemma \ref{Hall}. We keep the same notation as in Chapter \ref{Commutative
Graphs} and write $\Gamma(x)$ for the neighborhood of $x$ in the
bipartite graph, which is precisely $\im_{G_{h+1}}^{-1}(x)$. Let
$\pi_2$ be the projection onto $R_1$. For any $S\subseteq
W_1\times \im_{R_1}(v)$ we have from Lemma \ref{One to k map}
\begin{eqnarray*}
|\Gamma(S)| & = & \left| W_0 \times  \bigcup_{x\in \pi_2(S)}  \im^{-1}_{R_1}(x) \right | \\
            & \geq & \sum_{x\in \pi_2(S)} |W_0|\, |\theta(x)|\\
            & = & |\pi_2(S)|\, |W_0|\, k \\
            & = & |\pi_2(S)|\,|W_1| \\
            & \geq & |S|.
\end{eqnarray*}
Hence Hall's condition is satisfied and as a consequence so is
Pl\"{u}nnecke's.
\end{proof}
We construct arbitrarily long $R_k$ (and hence finish the proof of Theorem \ref{Sharp Commutative for Finite}) as follows. We start with
the two layer (and hence non-commutative) graph consisting of a
single vertex in $V_0$ joined to all $k$ vertices in $V_1$. A
first application of Proposition \ref{Multilayer Extension} yields an
$R_{k}$ of level two. Repeated applications yield an arbitrarily long $R_{k}$. 

In light of Theorem \ref{Sharp Commutative for Infinite} it should be noted that this construction does not lead to infinite regular commutative graphs as in each step the size of the bottom layer increases.

\subsection{Infinite regular commutative graphs}

\label{Infinite Regular Commutative Graphs}

The construction of arbitrarily long regular commutative graphs we have presented does not give infinite regular commutative graphs. This doesn't of course rule out their existence. In order to prove Theorem
\ref{Sharp Commutative for Infinite} we will examine how much a Pl\"{u}nnecke
graph originating at a singleton can grow. Pl\"{u}nnecke's
inequality gives $|V_h| \leq |V_1|^h$, but the growth is in fact
far from exponential.
\begin{lemma}\label{Growth m=1}
Let $G$ be an infinite commutative graph where $|V_0|=1$ and
$|V_1|=n$. Then
$$ |V_h| \leq \tbinom{n+h-1}{h}. $$
The bound is best possible.
\end{lemma}
\begin{proof}
We perform a double induction on $n$ and $h$. Let $A(n,h)$
be the maximum of $|V_h|$ taken over all commutative graphs with
$|V_0|=1$ and $|V_1|=n$. Take such a $G$ with $V_0=\{u\}$ and
$V_1=\{v_1,\dots,v_n\}$. Any element of $V_h$ can either be
reached from a path passing from $v_1$ or exclusively via paths
that pass from $\{v_2,\dots,v_n\}$. In the former case the vertex
lies in the commutative graph consisting of all paths that start
in $v_1$. By Pl\"{u}nnecke's upward condition the second layer of this
graph has at most $n$ elements and so there are at most $A(n,h-1)$
such vertices in $V_h$. In the later case the vertex lies in the
commutative graph consisting of all paths that start in $u$ and
end in $V_h \backslash \im^{(h)}(v_1)$. The second layer of this
graph is a subset of $\{v_2,\dots, v_n\}$ and hence there are at
most $A(n-1,h)$ such vertices in $V_h$. We have therefore proved
that $$ A(n,h) \leq A(n,h-1) + A(n-1,h).$$
It follows from Pl\"{u}nnecke's condition that $A(1,h)=1 = {h
\choose h}$ for all $h$ and we know that $A(n,1)=n= {n \choose
1}$. The stated bound follows inductively from the well known
identity ${a\choose b} = {a-1 \choose b} + {a-1 \choose b-1}$ and
is attained when $G$ is an independent addition graph on $n$
generators.
\end{proof}
Deducing Theorem \ref{Sharp Commutative for Infinite} is straightforward.
\begin{proof}[Proof of Theorem \ref{Sharp Commutative for Infinite}]
A path is an infinite commutative graph whose magnification ratios are all equal to one. 

Let $1\neq C\in \mathbb{Q}^+$ and $G$ be a commutative graph where $D_i(G)=C^i$ for all $i$. We have to show that $G$ is finite. 

When $C<1$ we let $V_0$ be the bottom layer of $G$. The definition of magnification ratios implies that there exists $\emptyset\neq Z_i \subseteq V_0$ such that $|\im^{(h)}(Z_i)|=C^i |Z_i|$. The quantity $C^i |Z_i|$ is a non-zero integer less than $C^i |V_0|$ and so $i$ cannot be arbitrarily large. When $C>1$ we let $V_1$ be the second layer of $G$. Lemma \ref{Growth m=1} implies that $$D_i(G) \leq \binom{|V_1|+i-1}{i} = O( i^{|V_1|}),$$ which for large enough $i$ is less than $C^i$. 
\end{proof}

\section[Inverse Theorem for Pl\"{u}nnecke's Inequality]{Inverse Theorem for Pl\"{u}nnecke's Inequality}
\label{Inverse}

\vspace{1ex}

We conclude the paper by establishing a necessary condition for Pl\"unnecke's inequality to be attained. Let us first recall a definition from Section \ref{Commutative Graphs}.
\begin{definition}
Let $Z\subseteq V_0$. The \textit{channel} of $Z$ is the commutative subgraph which consists of all paths of maximum length that start in $Z$.
\end{definition}
We use some of the results in Section \ref{Proof} to prove an inverse result for Theorem \ref{Plunnecke}. 
\begin{theorem}\label{Inverse Theorem for Plunnecke Inequality}
Let $C\in \mathbb{Q}^+$ and $G$ be a commutative graph with
$D_i(G)=C^i$ for all $i$. Then exists $\emptyset\neq Z\subseteq V_0$ whose channel is a regular commutative graph of ratio $C$.
\end{theorem}

\subsection{Inverse theorem for Pl\"{u}nnecke's inequality}

\label{Inverse Plunnecke}

The first step in proving Theorem \ref{Inverse Theorem for
Plunnecke Inequality} is to identify $Z$. It turns out that
choosing the smallest non-empty subset of $V_0$ that has a chance of working will do. We
will later need the cardinalities of the various layers of the
channel of such a $Z$.

\begin{lemma}\label{Cardinality of Z_i}
Let $C\in \mathbb{Q}^+$ and suppose $G$ is a commutative graph with
$D_i(G)=C^i$ for all $1\leq i \leq h$. Let $\emptyset\neq Z\subseteq V_0$ be of
minimal size subject to $|\im(Z)| = C |Z|$ and $H$ be the channel of $Z$ with vertex set $U_0\cup\dots\cup U_h$. Then
$|U_i|=C^i|U_0|$ for all $1\leq i \leq h$.
\end{lemma}
\begin{proof}
For any $S\subseteq Z=U_0$ we have that $\im^{(i)}(S) = \im^{(i)}_H(S)$ so we
will drop the subscript. Observe that $D_1(H)=C$. We use this to
show that $D_i(H)=C^i$ for all $i$. Indeed
$$ C^i=D_i(G)\leq D_i(H) \leq D_1^i(H) = C^i$$
the first inequality following from the definition of
magnification ratios, while the second from Theorem \ref{Plunnecke}. Hence $|\im^{(i)}(S)| = C^i |S|$ for some $S\subset
U_0$. $\im(S)^c\cup \im^{(i)}(S)$ is a separating set in the
weighted version of $H$, where as usual $w(v)=C^{-i}$ for all
$v\in U_i$. By Corollary \ref{Min Weight} we know that
\begin{eqnarray*}
|U_0| & \leq & w(\im(S)^c)+w(\im^{(i)}(S)) \\
        & = & C^{-1}(|U_1| - |\im(S)|) +C^{-i}|\im^{(i)}(S)| \\
        & = &  |U_0| - C^{-1}|\im(S)| + |S|.
\end{eqnarray*}
Thus $|\im(S)|\leq C|S|$. The minimality of $Z$ implies that $S=Z=U_0$.
\end{proof}

We proceed with the proof of Theorem \ref{Inverse Theorem for
Plunnecke Inequality}. We will use Lemma \ref{Technical Three
Layer Pull Down} on page \pageref{Technical Three Layer Pull Down}
repeatedly to show that $H$ has to in fact be regular.

\begin{proof}[Proof of Theorem \ref{Inverse Theorem for Plunnecke
Inequality}]
Similarly to above we let $\emptyset\neq Z\subseteq V_0$ be of minimal size
subject to $|\im(Z)|= C |Z|$. Our goal is to prove that its channel $H$ is a regular commutative graph of
ratio $C$. We will not have to work in $G$ any further so to keep
the notation simple we will write $\im$ and $\im^{-1}$ instead of
$\im_H$ and $\im^{-1}_H$. Note however that in general
$\im^{-1}_G\neq \im^{-1}_H$.

We want to apply Lemma \ref{Technical Three Layer Pull Down} so we
let $U_0\cup U_1\dots\cup U_h$ be the vertex set of $H$ with the
usual weights $w(v)=C^{-i}$ for all $v\in U_i$. We partition $U_1$
into $X_1,\dots, X_k$ (where $d^-\rst{X_i}=i$) and
$X^\prime_1,\dots, X^\prime_{k^\prime}$ (where
$d^+\rst{X^\prime_i}=i$). We also partition $U_0$ and $U_2$
respectively into $Y^\prime_1,\dots, Y^\prime_{k^\prime}$ (where
$d^+\rst{Y^\prime_i}=i$) and $Y_1,\dots, Y_k$ (where
$d^-\rst{Y_i}=i$). To check that the condition of Lemma
\ref{Technical Three Layer Pull Down} is satisfied we observe that
$U_1$, which by Lemma \ref{Cardinality of Z_i} has weight $|U_0|$, is by Corollary \ref{Min Weight}
a separating set of minimum weight. For every $S\subseteq U_1$
both $S^c\cup \im(S)$ and $S^c \cup \im^{-1}(S)$ are separating
sets. The minimality of $w(U_1)$ implies that
$$ |\im(S)|\geq C|S|~~~\mathrm{and}~~~ |\im^{-1}(S)|\geq C^{-1}|S|.$$

Our first task will be to establish that $Y^\prime_{k^\prime}=U_0$
and that the outgoing degree in $U_0\cup U_1$ is $k^\prime$.
Suppose not. Then
$$ \bigcup_{i=1}^{k^\prime - 1} Y^\prime_i$$
is both non-empty and not the whole of $U_0$. By the minimality of
$Z$
$$ \left|\bigcup_{i=1}^{k^\prime - 1} X^\prime_i\right| =
\left|\im\left(\bigcup_{i=1}^{k^\prime - 1}
Y^\prime_i\right)\right| >C \left|\bigcup_{i=1}^{k^\prime - 1}
Y^\prime_i\right|. $$ On the other hand by Lemma \ref{Technical
Three Layer Pull Down} we know that
$$ \left|\bigcup_{i=1}^{k^\prime - 1} X^\prime_i\right| =
\sum_{i=1}^{k^\prime-1}|X^\prime_i| = C
\sum_{i=1}^{k^\prime-1}|Y^\prime_i| =C
\left|\bigcup_{i=1}^{k^\prime - 1} Y^\prime_i\right| .$$ So
$Y^\prime_{k^\prime}=U_0$ and by Lemma \ref{Technical Three Layer
Pull Down} $|X^\prime_{k^\prime}| = C |Y^\prime_{k^\prime}|=
|U_1|$, so $X^\prime_{k^\prime}=U_1$ and $d^+\rst{U_0\cup
U_1}=k^\prime$.

Next we establish that $X_{k}=U_1$ and that the incoming degree in
$U_1\cup U_2$ is $k$. Let $j$ be minimal subject to $Y_j\neq
\emptyset$. Let $R$ be the channel between $Z=U_0$ and $Y_j$.

We observe that $\im^{-1}(Y_j)=X_j$. This holds as by
Pl\"unnecke's downward condition $$\im^{(-1)}(Y_j)\subseteq
\bigcup_{i=1}^jX_i.$$ The choice of $j$ implies that 
$Y_i=\emptyset$, for $i<j$. By Lemma \ref{Technical Three Layer Pull Down}
we have $|X_i|=C^{-1}|Y_i|=0$ for all $i<j$. Thus
$\im^{(-1)}(Y_j)=X_j$ as claimed.

Thus $R_0=\im^{-1}(X_j)$, $R_1 = X_j$ and $R_2=Y_j$ are the layers
of $R$.
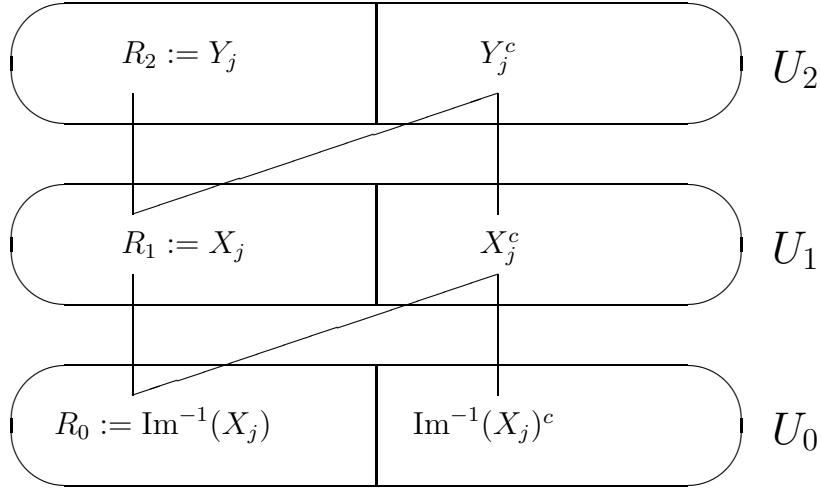
\begin{figure}[h t]
\begin{center}
\setlength{\unitlength}{0.8 cm}
 \begin{picture}(16,10)
 \put(14.5,7.7){\Large{$U_2$}}
 \put(8,8){\oval(12,2)}
 \put(8,7){\line(0,1){2}}
 \put(3.8,8){{$R_2 := Y_j$}}
 \put(9.7,8){{$Y_j^c$}}
 \put(4,5.5){\line(0,1){2}}
 \put(4,5.5){\line(3,1){6}}
 \put(10,5.5){\line(0,1){2}}
 \put(14.5,4.7){\Large{$U_1$}}
 \put(8,5){\oval(12,2)}
 \put(8,4){\line(0,1){2}}
 \put(3.8,4.85){{$R_1 : =X_j$}}
 \put(9.7,4.85){{$X_j^c$}}
 \put(4,4.5){\line(0,-1){2}}
 \put(10,4.5){\line(0,-1){2}}
 \put(10,4.5){\line(-3,-1){6}}
 \put(14.5,1.7){\Large{$U_0$}}
 \put(8,2){\oval(12,2)}
 \put(8,1){\line(0,1){2}}
 \put(2.7,1.85){{$R_0 := \im^{-1}(X_j)$}}
 \put(8.6,1.85){{$\im^{-1}(X_j)^c$}}
 \end{picture}
\end{center}
\caption{\label{Inverse Illustration}\small{An illustration of how
different parts of the graph are connected. Lines may correspond
to multiple or no edges.}}
\end{figure}
We will apply Lemma \ref{Technical Three Layer Pull Down Prime} on
page \pageref{Technical Three Layer Pull Down Prime} to $R$ and so
we need to check that the two conditions are satisfied. We begin
with the second. We have $d^-_R(v)=d^-_H(v)=j$ for all $v\in R$
and by Lemma \ref{Technical Three Layer Pull Down} that
$|R_2|=|Y_j| = C |X_j| = C |R_1|$. Thus
\begin{eqnarray*}
|E(R_1,R_2)| = \sum_{w\in R_2} d^-(w) = |R_2|\, j = C |R_1| \,j =
C |E(R_0,R_1)|.
\end{eqnarray*}
For the first we observe that $\im^{-1}_R(v)=\im^{-1}(v)$ for all
$v\in R$. We have seen above that $U_1$ is a separating set in $H$
of minimum weight and so we have that
$|\im^{-1}_R(S)|=|\im^{-1}(S)|\geq C^{-1} |S|$ for all $S\subseteq
R_1$. We can now apply Lemma \ref{Technical Three Layer Pull Down
Prime} to get:
\begin{eqnarray}\label{Size of Downward Graph}
|R_1| = C |R_0|.
\end{eqnarray}
On the other hand we know that $\im(\im^{-1}(X_j)^c)=X_j^c$ and so
if $R_0=\im^{-1}(X_j) \neq U_0$, the minimality of $Z$ implies
$$|U_1| - |X_j| >  C (|U_0| - |\im^{-1}(X_j)| = |U_1| - C |\im^{-1}(X_j)|$$
i.e. that $C |R_0|> |R_1|$, which contradicts (\ref{Size of
Downward Graph}). We must therefore have $R_0=U_0$. Hence
$|X_j|=|R_1|=C|R_0|=C|U_0|=|U_1|$, i.e. $X_j=U_1=X_k$ and so
$|Y_j|=C|X_j|=|U_2|$, i.e. $Y_j=U_2=Y_k$. In particular
$d^-\rst{U_1\cup U_2}=k$.

We therefore have regularity in the bottom three layers. We must
check that the ratio of $k^\prime$ to $k$ is $C$. This follows
from counting the edges between $U_0$ and $U_1$ in two ways:
$$ k^\prime |U_0|  = |E(U_0,U_1)| = k |U_1| = k C |U_0|.$$
The final step is to establish regularity for the remaining layers
of $G$. We consider any $w\in U_2$. $d^+(w)\leq k^\prime=C k$ and
so $$C |E(U_1,U_2)| = C |U_1| C k =  C k |U_2| \geq
|E(U_2,U_3)|.$$ The fact that $|U_2|=C|U_1|$ follows from by Lemma
\ref{Cardinality of Z_i}. Similarly $d^-(x)\geq k$ for any $x\in
U_3$ and so $$C |E(U_1,U_2)| = C k |U_2| = k |U_3| \leq
|E(U_2,U_3)|.$$ We must therefore have equality in each step and
therefore $d^+(w)= C k$ for all $w\in U_2$ and $d^-(x)=k$ for all
$x\in U_3$. We repeat this step for all remaining layers to finish
off the proof.
\end{proof}

\begin{remark}
An alternative way to prove Theorem \ref{Inverse Theorem for
Plunnecke Inequality} is to first establish the special case when
$C=1$ using Proposition \ref{Ruzsa} and then deduce the general
case by the multiplicativity of magnification ratios and degrees.
This time independent addition graphs cannot work and we need to
use regular commutative graphs.
\end{remark}

Proposition \ref{Ruzsa} gives a sensible looking
necessary and sufficient condition for all magnification ratios of
a commutative graph to equal one.

\begin{corollary}\label{N&SforC=1}
Let $G$ be a commutative graph and $V_0$ be its bottom layer. $D_i(G)=1$ for all $i$ if and only if there exist $|V_0|$ vertex
disjoint paths of maximum length in $G$ and the channel of some $\emptyset\neq Z\subseteq V_0$ is an $R_{1}$.
\end{corollary}

\begin{proof}
When $D_i(G)=1$ for all $i$ we know from Proposition \ref{Ruzsa}
that there are $|V_0|$ vertex disjoint paths of maximum length in $G$ and just
proved the existence of a suitable non-empty $Z\subseteq V_0$. Conversely
the existence of the vertex disjoint paths of maximum length guarantees that
$D_i(G)\geq 1$ for all $i$ and Lemma \ref{Cardinality of Z_i} guarantees that  $|\im^{(i)}(Z)|=|Z|$ and hence
$D_i(G)\leq 1$.
\end{proof}

Not a whole lot more can be said about the structure of such $G$.
It is clear that the channel of $Z^c$ must
have magnification ratios no smaller than one and that is about
it. For example take any commutative graph $G$ of level two with
vertex set $V_0\cup V_1 \cup V_2$ and $D_i(G)\geq 1$ and any
regular commutative graph $R$ with ratio one and vertex set
$U_0\cup U_1 \cup U_2$. Form a new graph $G^\prime$ of level two
by placing an edge between any $u\in V_i$ and any $v\in U_{i+1}$.
$G^\prime$ has magnification ratios equal to one as
$$|\im_{G^\prime}^{(i)}(S)|= |\im_{G}^{(i)}(S\cap V_0)|+
|\im_{R}^{(i)}(S\cap U_0)|\geq |S\cap V_0|+|S\cap U_0|=|S|$$ and
$$|\im_{G^\prime}^{(i)} (U_0)| = |\im_{R}^{(i)} (U_0)|=|U_0|. $$
$G^\prime$ is furthermore a commutative graph. The way $G$ is
joined to $R$ means that for the upward condition we only need to
worry about elements in $V_0$. Let $uv\in E(V_0,V_1)$. Then
$\im_{G^\prime}(v)= \im_G(v) \cup U_2$ and $\im_{G^\prime}(u) =
\im_G(u) \cup U_1$. We know from Pl\"{u}nnecke's condition that a
matching exists from $\im_G(v)$ to $\im_G(u)$ and from Proposition
\ref{Ruzsa} and Lemma \ref{Cardinality of Z_i} that a matching
exists from $U_2$ to $U_1$. Putting the two matchings together
gives a matching from $\im_{G^\prime}(v)$ to $\im_{G^\prime}(u)$.
Next take $uv\in E(V_0,U_1)$, $\im_{G^\prime}(v) = \im_{R}(v)$ and
we know from Proposition \ref{Ruzsa} that there is a matching from
$\im_R(v)$ to $U_1\subseteq \im_{G^\prime}(u)$ and hence from
$\im_{G^\prime}(v)$ to $\im_{G^\prime}(u)$. Similar considerations
show that the downward condition is satisfied.

\bibliography{all}

\begin{thebibliography}{10}

\bibitem{Hall}
P.~Hall.
\newblock On representatives of subsets.
\newblock {\em J. London Math. Soc.}, 10:26--30, 1935.

\bibitem{Malouf1995}
{J.\,L.} Malouf.
\newblock On a theorem of {P}l{\"u}nnecke concerning the sum of a basis and a
  set of positive density.
\newblock {\em J. Number Theory}, 54:12--22, 1995.

\bibitem{Menger}
K.~Menger.
\newblock Zur allgemeinen kurventheorie.
\newblock {\em Fund. Math.}, 10:96--115, 1927.

\bibitem{Nathanson1996}
{M.\,B.} Nathanson.
\newblock {\em Additive Number Theory: Inverse Problems and the Geometry of
  Subsets}.
\newblock Springer, New York, 1996.

\bibitem{Plunnecke1970}
H.~Pl{\"{u}}nnecke.
\newblock Eine zahlentheoretische anwendung der graphtheorie.
\newblock {\em J. Reine Angew. Math.}, 243:171--183, 1970.

\bibitem{Ruzsa1989}
{I.\,Z.} Ruzsa.
\newblock An application of graph theory to additive number theory.
\newblock {\em Scientia, Ser. A}, 3:97--109, 1989.

\bibitem{Ruzsa1991}
{I.\,Z.} Ruzsa.
\newblock Addendum to: An application of graph theory to additive number
  theory.
\newblock {\em Scientia, Ser. A}, 4:93--94, 1990/1991.

\bibitem{Ruzsa2009}
{I.\,Z.} Ruzsa.
\newblock Sumsets and structure.
\newblock In {\em Combinatorial Number Theory and Additive Group Theory}.
  Springer, New York, 2009.

\bibitem{TaoNotes}
T.~Tao.
\newblock Additive combinatorics.
\newblock \url{http://www.math.ucla.edu/~tao/254a.1.03w/notes1.dvi}.

\bibitem{Tao-Vu2006}
T.~Tao and {V.\,H.} Vu.
\newblock {\em Additive Combinatorics}.
\newblock Cambridge University Press, Cambridge, 2006.

\end{thebibliography}

$\hspace{12pt}$\textsc{Department of Pure Mathematics and
Mathematical Statistics} \\ \textsc{Wilberforce Road, Cambridge CB3 0WB,
England}

$\hspace{12pt}$ \textit{Email address}: giorgis@cantab.net
\end{document}